\documentclass[12pt]{amsart}

\usepackage{amsmath,amssymb}
\usepackage{amsfonts}
\usepackage{amsthm}
\usepackage{latexsym}
\usepackage{graphicx}
\usepackage{epsfig}

\newtheorem{theorem}{Theorem}[section]

\newtheorem{corollary}[theorem]{Corollary}

\newtheorem{prop}{Proposition}[section]
\newtheorem{rema}[prop]{Remark}

\makeatletter \@addtoreset{equation}{section} \makeatother

\def\ppt{\frac{\partial}{\partial t}}

\def\p{\partial}

\def\({\left(}
\def\){\right)}

\begin{document}

\title{Differential Harnack Estimates for
Parabolic Equations}

\author{Xiaodong Cao $^*$ }
\thanks{$^*$ Research
partially supported by NSF grant DMS~0904432.}

\address{Department of Mathematics,
 Cornell University, Ithaca, NY 14853-4201, USA.}
\email{cao@math.cornell.edu}

\author{Zhou Zhang $^\dagger$
}
\thanks{$^\dagger$ Research
partially supported by NSF grant DMS 0904760.}

\address{Department of Mathematics, University of
Michigan, at Ann Arbor, MI 48109, USA.} \email{zhangou@umich.edu}

\renewcommand{\subjclassname}{%
  \textup{2000} Mathematics Subject Classification}

\date{Jan. 8, 2010}

\maketitle

\markboth{Xiaodong Cao, Zhou Zhang} {Differential Harnack
Estimates for Parabolic Equations}

\begin{abstract}  Let $(M,g(t))$ be a
solution to the Ricci flow on a closed Riemannian manifold. In
this paper, we prove differential Harnack inequalities for
positive solutions of nonlinear parabolic equations of the type
$$\ppt f=\Delta f-f \ln f +Rf.$$ We also comment on an earlier result
of the first author on positive solutions of the conjugate heat
equation under the Ricci flow.
\end{abstract}

\section{\textbf{Introduction}}

Let $(M,g(t))$, $t\in [0, T)$, be a solution to the Ricci flow on
a closed manifold $M$. In the first part of this paper, we deal
with positive solutions of nonlinear parabolic equations on $M$.
We establish Li-Yau type differential Harnack inequalities for
such positive solutions. To be more precisely, $g(t)$ evolves
under the Ricci flow
\begin{equation}
\frac{\partial g(t)}{\partial t}= -2Rc, \label{rf}
\end{equation} where $Rc$ denote the Ricci
curvature of $g(t)$. We first assume that the initial metric
$g(0)$ has nonnegative curvature operator, which implies that for
all time $t\in [0,T)$, $g(t)$ has nonnegative curvature operator
(c.f. \cite{HPCO}). Consider a positive function $f(x,t)$ defined
on $M\times [0,T)$, which solves the following nonlinear parabolic
equation,
\begin{equation}
\label{heat}
\frac{\partial f}{\partial t}=\triangle f-f\ln f +Rf.
\end{equation}
The symbol $\triangle$ here stands for the Laplacian of the
evolving metric $g(t)$ and $R$ is the scalar curvature of $g(t)$.
For simplicity, we omit $g(t)$ in the above notations. All
geometry
operators are with respect to the evolving metric $g(t)$.\\

Differential Harnack inequalities were originated by P. Li and
S.-T. Yau in \cite{ly86} for positive solutions of the heat
equation (therefore it is also known as Li-Yau type Harnack
estimates). The technique was then brought into the study of
geometric evolution equation by R. Hamilton (for example, see
\cite{harnack}) and ever since has been playing an important role
in the study of geometric flows. Applications include estimates on
the heat kernel; curvature growth control; understanding the
ancient solutions for geometric flows; proving noncollapsing
result in the Ricci flow
(\cite{perelman02}); etc. See \cite{ni08} for a recent survey
on this subject by L. Ni.\\


Using maximum principle, one can see that the solution for
(\ref{heat}) will remain positive along the flow.  It exists as
long as the solution for (\ref {rf}) exists. The study of the
Ricci flow coupled with a heat-type (or backward heat-type)
equation started from R. Hamilton's paper \cite{Hsurvey}.
Recently, there has been some interesting study on this topic. In
\cite{perelman02}, G. Perelman proved a differential Harnack
inequality for the fundamental solution of the conjugate heat
equation under the Ricci flow. In \cite{c08}, the first author
proved a differential Harnack inequality for general positive
solutions of the conjugate heat equation, which was also proved
independently by S. Kuang and Q. S. Zhang in \cite{kz08}. For the
Ricci flow coupled with the heat equation, the study was pursued
in \cite{ch2008, guenther02,ni04,zhang06}. Various estimates are
obtained recently by M. Bailesteanu, A. Pulemotov and the first
author in \cite{bcp09}, and by S. Liu in \cite{liu09}. For
nonlinear parabolic equations under the Ricci flow, local gradient
estimates for positive solutions of equation
$$\ppt f=\Delta f+af \ln f +b f,$$ where $a$ and $b$ are
constants,
has been studied by Y. Yang in \cite {yang2008}. For general
evolving metrics, similar estimate has been obtained  by A. Chau,
L.-F. Tam and C. Yu in \cite{cty2007}, by S.-Y. Hsu in
\cite{hsu2008}, and by J. Sun in \cite{sun09}.  In \cite{ma2006},
L. Ma proved a gradient estimate for the elliptic equation
$$\Delta f+af \ln f +bf=0.$$

\vspace{0.1in} \noindent In (\ref{heat}), if one defines
$$u(x,t)=-\ln f(x,t),$$
then the function $u=u(x,t)$ satisfies the evolution
equation
\begin{align}
\label{heat2}  \frac{\partial u}{\partial t}=
\Delta u- |\nabla u|^2-R-u.
\end{align}
The computation from (\ref{heat}) to (\ref{heat2}) is standard,
which
also gives the explicit relation between these two equations.\\


Our motivation to study (\ref{heat}) under the Ricci flow comes
from the geometric interpretation of (\ref{heat2}), which arises
from the study of expanding Ricci solitons. Recall that if we are
on a gradient expanding Ricci soliton $(M, g)$ satisfying
$$R_{ij}+\nabla_i \nabla_j w=-\frac14 g_{ij},$$
where $w$ is called soliton potential function, then we have
$$R(g)+\Delta_g w=-\frac{n}{4}.$$
In sight of this, by taking covariant derivative for the soliton
equation and applying the second Bianchi identity, one can see
that
$$R(g)+|\nabla_g w|_g^2+\frac{w}{2}=\mbox{constant}.$$
Also notice that the Ricci soliton potential function $w$ can be
differed by a constant in the above equations. So by choosing this
constant carefully, we have
$$R(g)+|\nabla_g w|_g^2=-\frac{w}{2}-\frac{n}{8}.$$
One consequence of the above identities
is the following
\begin{equation}
\label{intuition}
|\nabla_g w|_g^2=\Delta_g w- |\nabla_g w|_g^2-
R(g)-w.
\end{equation}\\

Recall that the Ricci flow solution for an expanding soliton is
$g(t)=c(t)\cdot\phi(t)^*g$ (c.f. \cite{chowetc1}), where $c(t)=1+
\frac{t}{2}$ and the family of diffeomorphism $\phi(t)$ satisfies,
for any $x\in M$,
$$\frac{\partial }{\partial t}(\phi(t)(x))=\frac{1}{c(t)}\cdot(\nabla
_g w)\(\phi(t)(x)\).$$ Thus the corresponding Ricci soliton
potential $\phi(t)^*w$ satisfies
$$\frac{\p \phi(t)^*w}{\p t}(x)=\frac{1}{c(t)}(\nabla_g
w)(w)\(\phi(t)(x)\)=|\nabla\phi(t)^*w|^2(x).$$ Along the Ricci
flow, (\ref{intuition}) becomes
$$|\nabla \phi^*w|^2=\Delta \phi^*w- |\nabla \phi^*w
|^2-R-\frac{\phi^*w}{c(t)}.$$ Hence the evolution equation for the
Ricci soliton potential is \begin{equation} \frac{\p\phi(t)^*w}{\p
t}=\Delta \phi^*w- |\nabla
\phi^*w|^2-R-\frac{\phi^*w}{c(t)}.\end{equation} The second
nonlinear parabolic equation that we investigate in this paper is
\begin{align}
\label{heat3}  \frac{\partial u}{\partial t}=\Delta u-|\nabla
u|^2-R-\frac{u}{1+\frac{t}{2}}.
\end{align}
Notice that (\ref{heat2}) and (\ref{heat3}) are closely related
and only differ by their last terms.\\


Our first result deals with (\ref{heat}) and (\ref{heat2}).
\begin{theorem}
\label{theorem1.1} Let $(M, g(t))$, $t\in [0, T)$, be a solution
to the Ricci flow on a closed manifold, and suppose that $g(0)$
{\em (}and so $g(t)${\em )} has weakly positive curvature
operator. Let $f$ be a positive solution to the heat equation {\em
(\ref{heat})}, $u=-\ln f$ and
\begin{equation}\label{defH}
H=2 \triangle u-|\nabla u|^2-3R-\frac{2n} {t}.
\end{equation} Then
for all time $t\in (0,T)$
$$H\leqslant\frac{n}{4}.$$
\end{theorem}

\begin{rema}
The result can be generalized to the context of $M$ being
non-compact. In order for the same argument to work, we need to
assume that the Ricci flow solution $g(t)$ is complete with the
curvature and all the covariant derivatives being uniformly
bounded {\em (}in the space direction{\em )}.
\end{rema}

Our next result deals with (\ref{heat3}), which is also a natural
evolution equation to consider with, by previous discussion.

\begin{theorem}
\label{theorem1.2} Let $(M, g(t))$, $t\in [0, T)$, be a solution
to the Ricci flow on a closed manifold, and suppose that $g(0)$
{\em (}and so $g(t)${\em )} has weakly positive curvature
operator. Let $u$ be a smooth solution to {\em (\ref{heat3})},
define
\begin{equation}\label{defH}
H=2 \triangle u-|\nabla u|^2-3R-\frac{2n} {t}.
\end{equation} Then
for all time $t\in (0,T)$
$$H\leqslant 0.$$
\end{theorem}

\begin{rema}
If $f$ is a positive function such that $f=e^{-u}$, then $f$
satisfies the following evolution equation $$ \frac{\partial
f}{\partial t}=\triangle f+Rf-\frac{f\ln f}{1+\frac{t}{2}}.
$$
\end{rema}
\vspace{0.1in}

In \cite{c08}, the first author studied the conjugate heat
equation under the Ricci flow. In particular, the following
theorem was proved.

\begin{theorem} \label{thmnph} \cite[Theorem 3.6]{c08} Let $(M, g(t))$,
$t\in [0,T]$, be a solution to the
Ricci flow, suppose that $g(t)$ has nonnegative scalar curvature.
Let $f$ be a positive solution of the conjugate heat equation
$$\ppt f=-\triangle f
+Rf,$$ let $v=-\ln f-\frac{n}{2}\ln (4\pi \tau)$, $\tau=T-t$ and
$$P=2 \triangle v-|\nabla v|^2+R-\frac{2n}{\tau}.$$ Then we have
\begin{align}\label{harnack}
\frac{\partial}{\partial \tau} P=&\triangle P-2\nabla P \cdot
\nabla v-2|v_{ij}+R_{ij}-\frac{1}{\tau}g_{ij}|^2
-\frac{2}{\tau}P-2\frac{|\nabla v|^2}{\tau}-2\frac{R}{\tau}.
\end{align}Moreover, for all time $t\in [0,T)$,
$$P\leqslant 0.$$
\end{theorem}

 In the last section, we apply a
similar trick as in the proof of  Theorem \ref{theorem1.1} and
obtain a slightly different result, where we no longer need to
assume that
$g(t)$ has nonnegative scalar curvature.\\

{\bf Acknowledgments} Xiaodong Cao wants to thank the organizers
of the conference ``Complex and Differential Geometry" for their
invitation and hospitality. Both authors want to express their
gratitude to East China Normal University, where they started this
discussion.

\section{\textbf{Proof of Theorem \ref{theorem1.1} and
Application}}

The evolution equations of $u$, is very similar to what is
considered in \cite{ch2008}. So the computation for the very
general setting there can be applied. The only difference is now
we have more terms coming from time derivative $\ppt u$.

\begin{proof}[Proof of Theorem \ref{theorem1.1}]
Recall the definition of $H$ from (\ref{defH}), comparing with
\cite[Corollary 2.2]{ch2008}, we have
\begin{align}
\label{harnack}
\ppt H=
&\triangle H-2\nabla H \cdot \nabla u-2|u_{ij}-
R_{ij}-\frac{1}{t} g_{ij}|^2-\frac{2}{t}H-\frac{2}
{t}|\nabla u|^2\\ \nonumber
&- 2\(\ppt R+\frac{R}{t}+2\nabla R \cdot \nabla
u+2 R_{ij}u_iu_j\) -2\Delta u+2|\nabla u|^2,
\end{align}
where the last two terms of the right hand side coming from the
extra term $-u$ in (\ref {heat2}). Plugging in $-2\Delta
u+2|\nabla u|^ 2=-H+|\nabla u|^2-3R-\frac{2n}{t}$, one arrives at
\begin{align}
\ppt H=
&\triangle H-2\nabla H \cdot \nabla u-2|u_{ij}-
R_{ij}-\frac{1}{t}g_{ij}|^2 -\(\frac{2}{t}+1\)H\\
\nonumber
&+\(1-\frac{2}{t}\)|\nabla u|^2-3R-\frac{2n}{t}-
\(\ppt R +\frac{R}{t}+2\nabla R \cdot\nabla u+
2 R_{ij}u_iu_j\).
\end{align}
In sight of the definition of $H$ (\ref{defH}), for $t$ small
enough, we have $H<0$. Since $g_{ij}$ has weakly positive
curvature operator, by the trace Harnack inequality for the Ricci
flow proved by R. Hamilton in \cite{harnack}, we have
$$\ppt R +\frac{R}{t}+2 \nabla R \cdot \nabla
u+2 R_{ij}u_iu_j\geqslant 0.$$ Also we have $R\geqslant 0$. Notice
that the term $\(1-\frac{2}{t}\)|\nabla u|^2$ prevents us from
obtaining an upper bound for $H$ for $t>2$.\\


We can deal with this by the following simple manipulation. To
begin with, one observes that from the definition of $H$,
$$|\nabla u|^2=2\(\Delta u-R-\frac{n}{t}\)-H-R.$$
We also have the following equality from
definition,
$$tr\(u_{ij}-R_{ij}-\frac{1}{t}g_{ij}\)=\Delta u-R-
\frac{n}{t}.$$ Now we can continue the computation for the
evolution of $H$ as follows,
\begin{align*}
\ppt H
\leqslant&\triangle H-2\nabla H \cdot \nabla
u-2|u_{ij}-R_{ij}-\frac{1}{t}g_{ij}|^2 -\(\frac{2}
{t}+1\)H-\frac{2}{t}|\nabla u|^2\\ \nonumber
&-4R+2\(\Delta u-R-\frac{n}{t}\)-H-\frac{2n}
{t}\\ \nonumber
\leqslant&\triangle H-2\nabla H \cdot \nabla
u-\frac{2}{n}\(\Delta u-R-\frac{n}{t}\)^2-\(\frac
{2}{t}+1\)H-\frac{2}{t}|\nabla u|^2\\
&-4R+2\(\Delta u-R-\frac{n}{t}\)-H-\frac{2n}
{t}\\ \nonumber
=&\triangle H-2\nabla H \cdot \nabla u-\(\frac
{2}{t}+2\)H-\frac{2}{t}|\nabla u|^2-4R-\frac{2n}
{t}\\
&-\frac{2}{n}\(\Delta u-R-\frac{n}{t}-\frac{n}
{2}\)^2+\frac{n}{2}\\
\leqslant&\triangle H-2\nabla H \cdot \nabla
u-\(\frac{2}{t}+2\)H-\frac{2}{t}|\nabla u|^2-4R
-\frac{2n}{t}+\frac{n}{2}.
\end{align*}
The essential step is the second inequality where we make use of
the elementary inequality
$$|u_{ij}-R_{ij}-\frac{1}{t}g_{ij}|^2\geqslant
\frac{1}{n}\(\Delta u-R-\frac{n}{t}\)^2.$$

\vspace{0.1in}

Now we can apply maximum principle. The value of $H$ for very
small positive $t$ is clearly very negative. So we only need to
consider the maximum value point is at $t> 0$ for the desired
estimate.\\


For $\forall T_0 < T$, assume that the maximum in $(0, T_0]$ is
taken at $t_0>0$. At the maximum value point, using the
nonnegativity of $|\nabla u|^2$ and $R$, one has
$$H\leqslant\frac{-4n+nt_0}{4+4t_0}=\frac
{n}{4}\(1-\frac{5}{t_0+1}\)\leqslant \frac{n}
{4}\(1-\frac{5}{T+1}\).$$ So if $T\leqslant 4$, i.e., for time in
$[0, 4)$, $H\leqslant 0$. In general, we have
$$H\leqslant\frac{n}{4}.$$

Theorem \ref{theorem1.1} is thus proved.\end{proof}
As a
consequence of Theorem \ref {theorem1.1}, we have

\begin{corollary}
\label{intH} Let $\(M, g(t)\)$, $t\in [0, T)$, be a solution to
the Ricci flow on a closed manifold, and suppose that $g(0)$ {\em
(}and so $g(t)${\em )} has weakly positive curvature operator. Let
$f$ be a positive solution to the heat equation
$$\ppt f=\Delta f -f \ln f+Rf.$$
Assume that $(x_1, t_1)$ and $(x_2, t_2)$,
$0<t_1<t_2$, are two points in $M\times
(0, T)$. Let $$\Gamma=\inf_{\gamma}\int
_{t_1}^{t_2} e^t\(|\dot{\gamma}|^2+R+\frac
{2n}{t}+\frac{n}{4}\)dt,$$ where $\gamma$
is any space-time path joining $(x_1, t_1)$
and $(x_2, t_2)$. Then we have
$$e^{t_1} \ln f(x_1,t_1)\leqslant e^{t_2} \ln
f(x_2, t_2) +\frac{\Gamma}{2}.$$
\end{corollary}

\vspace{0.1in}

This inequality is in the type of classical Harnack inequalities.
The proof is quite standard by integrating the differential
Harnack inequality.
We include it here for completeness.\\

\begin{proof}
Pick a space-time curve connecting $(x_1, t_1)$ and $(x_2, t_2)$,
$\gamma(t)=\(x(t), t\)$ for $t\in [t_1, t_2]$. Recall that  $u (x,
t)=-\ln f(x, t)$. Using the evolution equation for $u$, we have
\begin{equation}
\begin{split}
\frac{d}{d t}u\(x(t), t\)
&= \frac{\p u}{\p t}+\nabla u\cdot\dot
{\gamma} \\
&\leqslant \Delta u-|\nabla u|^2-R-u+
\nabla u\cdot\dot{\gamma} \\
&\leqslant \Delta u-\frac{|\nabla u|^2}
{2}-R-u+\frac{|\dot{\gamma}|^2}{2}.
\end{split}
\end{equation}
Now by Theorem \ref{theorem1.1}, we have
$$\Delta u=\frac{1}{2}\(H+|\nabla u|^2+3R+
\frac{2n}{t}\)\leqslant \frac{1}{2}\(\frac{n}{4}+ |\nabla
u|^2+3R+\frac{2n}{t}\).$$ So we have the following estimation,
$$\frac{d}{d t}u\(x(t), t\)\leqslant \frac{1}{2}
\(|\dot{\gamma}|^2+R+\frac{2n}{t}+\frac{n} {4}\)-u.$$ For any
space-time curve $\gamma$, we arrives at
$$\frac{d}{d t}(e^t\cdot u)\leqslant \frac{e^t}
{2}\(|\dot{\gamma}|^2+R+\frac{2n}{t}+\frac {n}{4}\).$$

\vspace{0.1in}

Hence the desired Harnack inequality is proved by integrating $t$
from $t_1$ to $t_2$.\end{proof}

\section{\textbf{Proof of Theorem \ref{theorem1.2}}}

In this section we study $u$ satisfying the evolution equation
(\ref{heat3}) originated from gradient expanding Ricci soliton
equation. We investigate the same quantity $$H=2 \triangle u-
|\nabla u|^2-3R-\frac{2n} {t}$$ as in the last section. The
evolution equations of $u$, is still very similar to what is
considered in \cite{ch2008}. We will have slightly different terms
coming from time derivative $\ppt u$ when computing the evolution
equation satisfied by $H$. Comparing with \cite[Corollary
2.2]{ch2008}, we proceed as follows.
\begin{proof}[Proof of Theorem \ref{theorem1.2}]
\begin{align}
\label{harnack} \ppt H= &\triangle H-2\nabla H \cdot \nabla
u-2|u_{ij}- R_{ij}-\frac{1}{t} g_{ij}|^2-\frac{2}{t}H-\frac{2}
{t}|\nabla u|^2\\ \nonumber &- 2\(\ppt R+\frac{R}{t}+2\nabla R
\cdot \nabla u+2 R_{ij}u_iu_j\)+\frac{2}{t+2}\(-2\Delta u+
2|\nabla u|^2\),
\end{align}
where the last two terms of the right hand side come from the
extra term $-\frac{u}{1+\frac{t} {2}}$ in (\ref {heat3}). Plugging
in $-2\Delta u+ 2|\nabla u|^ 2=-H+|\nabla u|^2-3R-\frac{2n}{t}$,
one arrives at
\begin{align}
\ppt H= &\triangle H-2\nabla H \cdot \nabla u-2|u_{ij}-
R_{ij}-\frac{1}{t}g_{ij}|^2 -\(\frac{2}{t}+\frac{2} {t+2}\)H-\frac
{6}{t+2}R\\
\nonumber &+\(\frac{2}{t+2}-\frac{2}{t}\)|\nabla
u|^2-\frac{4n}{t^2+2t}-\(\ppt R +\frac{R} {t}+2\nabla R
\cdot\nabla u+2 R_{ij}u_iu_j\).
\end{align}
In sight of the definition of $H$, for $t$ small enough, we have
$H<0$. Since $g(t)$ has weakly positive curvature operator, by the
trace Harnack inequality for the Ricci flow (\cite{harnack}), we
have
$$\ppt R +\frac{R}{t}+2 \nabla R \cdot \nabla
u+2 R_{ij}u_iu_j\geqslant 0.$$

Notice that now the coefficient for $|\nabla u |^2$ on the right
hand side is $\frac{2}{t+2}- \frac{2}{t}<0$, and we have
$R\geqslant 0$. So one can conclude directly from Maximum
Principle that $H\leqslant 0$. \end{proof}

\section{\textbf{ A Remark on the Conjugate Heat Equation}}

In this section we point out a simple observation for
\cite[Theorem 3.6]{c08}. The assumption on scalar curvature is not
needed below. We follow the original set-up in \cite{c08} here.\\


Over a closed manifold $M\sp n$, $g(t)$, $t\in [0, T]$, is a
solution to the Ricci flow (\ref{rf}); $f(\cdot, t)$ is a positive
solution of the conjugate heat equation
\begin{equation}\label{che}
\frac{\p f}{\p t}=-\Delta f+Rf, \end{equation} where $\Delta$ and
$R$ are Laplacian and scalar curvature with respect to the
evolving metric $g(t)$. Notice that $\int_M f(\cdot, t) d
\mu_{g(t)}$ is a constant along the flow.\\

Set $$v=-\log f-\frac{n\log(4\pi\tau)}{2},$$  where $\tau =T-t$
and define
$$P:=2\Delta v-|\nabla v|^2+R-\frac{2n}{\tau}.$$
Now we can prove the following result which is closely related to
\cite[Theorem 3.6]{c08}.
\begin{theorem}\label{thm4.1} Let $\(M, g\(t\)\)$, $t\in[0,T]$, be a solution to the
Ricci flow on a closed manifold. $f$ is a positive solution to the
conjugate heat equation {\em (\ref{che})}, and $v$ is defines as
above. Then we have
$$\max_M ~(2\Delta v-|\nabla v|^2+R)$$ increases along the Ricci flow.
\end{theorem}
\begin{proof}
Following the computation as in \cite[Theorem 3.6]{c08}, one has
$$\frac{\p P}{\p \tau}=\Delta P-2\nabla P\cdot\nabla
v-2|\nabla^2 v+Rc-\frac{1}{\tau}g|^2-\frac{2}{\tau}P-
\frac{2}{\tau}|\nabla v|^2-\frac{2}{\tau}R.$$ Applying the
elementary inequality
$$|\nabla^2 v+Rc-\frac{1}{\tau}g|^2\geqslant
\frac{1}{n}\(\Delta v+R-\frac{n}{\tau}\)^2,$$ and noticing that
$$P+|\nabla v|^2+R=2\(\Delta v+R-\frac{n}{\tau}\),$$
we arrive at
\begin{equation*}
\begin{split}
\frac{\p P}{\p \tau}
&\leqslant \Delta P-2\nabla P\cdot\nabla v-\frac{1}{2n}
(P+|\nabla v|^2+R)^2-\frac{2}{\tau}(P+|\nabla v|^2+R) \\
&= \Delta P-2\nabla P\cdot\nabla v-\frac{1}{2n}\(P+ |\nabla
v|^2+R+\frac{2n}{\tau}\)^2+\frac{2n}{\tau^2}.
\end{split}
\end{equation*}
Thus if we define
$$\widetilde P:=P+\frac{2n}{\tau}=2\Delta v-|\nabla v|^2
+R,$$ we have
$$\frac{\p \widetilde P}{\p \tau}\leqslant \Delta \widetilde
P-2\nabla \widetilde P\cdot\nabla v.$$ Hence $\max_M(2\Delta
v-|\nabla v|^2+R)$ decreases as $\tau$ increases, which means that
it increases as $t$ increases. This concludes the
proof.\end{proof}

\begin{rema}
Notice that we do not need to introduce $\tau$ in {\em Theorem
\ref{thm4.1}}, but we keep the notation here so it is easy to
compare with \cite[Theorem 3.6]{c08}.
\end{rema}

\begin{rema}
{\em Theorem \ref{thm4.1}} and  \cite[Theorem 3.6]{c08} estimate
quantities differ by $\frac{2n}{\tau}$. Here we do not need to
assume nonnegative scalar curvature as in \cite[Theorem 3.6]{c08}.
Moreover, one can also prove this result for complete non-compact
manifolds with proper boundness assumption.
\end{rema}

\bibliographystyle{plain}
\bibliography{bio}
\end{document}